\documentclass[12pt]{amsart}
\usepackage{latexsym,amssymb}
\usepackage{vmargin}
\setmargrb{.9in}{.9in}{.9in}{.9in}

\title[$C^{\ast}$-Algebras and w$^{\ast}$fpp]{Separable $C^{\ast}$-Algebras and
weak$^{\ast}$-fixed point property}%
\author{Gero Fendler}%
\address{{Faculty of mathematics, University of Vienna}%
{Oskar-Morgenstern-Platz 1, 1090~Wien, Austria}}
\email{gero.fendler@univie.ac.at}  

\author{Michael Leinert}%
\address{{Institut f\"ur angewandte Mathematik, Universit\"at~Heidelberg}%
{Im Neuenheimer Feld, Geb\"aude 294, 69120~Heidelberg, Germany}}%
\email{leinert@math.uni-heidelberg.de}%



\newtheorem{theorem}{Theorem}[section]
\newtheorem{lemma}{Lemma}[section]
\newtheorem{prop}[theorem]{Proposition}
\theoremstyle{definition}
\newtheorem{definition}{Definition}[section]

\theoremstyle{remark}
\newtheorem{remark}{Remark}[section]
\theoremstyle{corollary}
\newtheorem{corollary}{Corollary}[section]
\newcommand{\norm}[2][]{\lVert{#2}\rVert_{#1}} 
\newcommand{\abs}[1]{\lvert{#1}\rvert}
\newcommand{\leb}{\operatorname{dx}}
\newcommand{\NN}{\ensuremath{\mathbb{N}}}
\newcommand{\CC}{\ensuremath{\mathbb{C}}}

\newcommand{\scr}{\mathcal}

\newcommand{\sign}{\operatorname{sign}}

\newcommand{\diam}[1]{\operatorname{diam}(#1)}
\newcommand{\sep}[1]{\operatorname{sep}(#1)}

\renewcommand{\limsup}{\operatorname{limsup}}

\begin{document}

\maketitle
\begin{abstract}
We show that the  spectrum $\widehat{A}$
of a  separable $C^{\ast}$-algebra $A$  is discrete
if and only if $A^{\ast}$,
the  Banach space dual of $A$, has the weak$^{\ast}$ fixed point property.
We prove further that these properties are equivalent among others to
the uniform weak$^{\ast}$ Kadec-Klee property of $A^{\ast}$
and to the coincidence of the weak$^{\ast}$ topology with 
the norm topology on the pure states of $A$.
If one assumes the set theoretic diamond axiom,
then the separability is necessary. 
\end{abstract}

\section{Introduction}
It is a well known theorem in harmonic analysis that a locally compact 
group $G$ is compact if and only if its dual $\widehat{G}$
is discrete. This dual is just the spectrum of 
the full $C^{\ast}$-algebra $C^{\ast}(G)$ of $G$.
(The spectrum of a $C^{\ast}$-algebra being the unitary equivalence classes of 
the irreducible $\ast$-representations endowed with the inverse 
image of the Jacobson topology on the set of primitive ideals.)
There is a bunch of 
properties of the weak$^{\ast}$ topology for the  Fourier--Stieltjes
algebra $B(G)$ of $G$, which are equivalent to the compactness of the group,
see~\cite{FLL}. Some of them, which can be formulated
in  purely $C^{\ast}$-algebraic terms, are the topic of this note.
\par
Let $E$ be a Banach space and $K$ be a non--empty bounded closed convex subset.
$K$ has the {\em fixed point property} if any non--expansive map
$T:K\to K$ (i.e. $\norm{Tx-Ty}\leq\norm{x-y}$ for all $x,y\in K$)
has a fixed point. We say that $E$ has the {\em weak fixed point property} 
if every weakly compact convex subset of $E$ has the fixed point property.
If $E$ is a dual Banach space, we consider the 
{\em weak$^{\ast}$ fixed point property} of $E$, i.e.\  the property that 
every weak$^{\ast}$ compact convex subset of $E$ has the fixed point property.
Since in a dual Banach space convex weakly compact sets are weak$^{\ast}$ 
compact, the weak$^{\ast}$ fixed point property of $E$ implies the 
weak fixed point property.
\par
As in \cite{FLL}
we shall consider the case of a left reversible semigroup
$\mathcal{S}$ acting by non-expansive mappings 
separately continuously on  a non-empty weak$^{\ast}$ compact convex set 
$K\subset E$. We say that 
$E$ has the 
{\em weak$^{\ast}$ fixed point property for left reversible semigroups} if 
under these conditions there always is a common fixed point in $K$.
\par
One of the main results of~\cite{FLL} is that a locally compact group $G$ is
compact if and only if $B(G)$ has 
the weak$^{\ast}$ fixed point property for non-expansive
maps, equivalently for left reversible semigroups. 
\par
We shall prove that a separable $C^{\ast}$-algebra has a discrete spectrum
if and only if its Banach space dual has the 
weak$^{\ast}$ fixed point property. 
We consider separable $C^{\ast}$-algebras only, 
because we use that a separable $C^{\ast}$-algebra with one point spectrum
is known to be isomorphic to the algebra of compact operators on some
Hilbert space~\cite{Ros73}. The converse, namely that the $C^{\ast}$-algebra
of the compact operators have up to unitary equivalence only one irreducible
representation, was proved by Naimark~\cite{Nai48}. His question~\cite{Nai51}
whether these are the only
$C^{\ast}$-algebras with a one point spectrum became known as Naimark's problem.
Assuming the set theoretic diamond axiom, independent from ZFC 
(Zermelo Frankel set theory
with the axiom of choice), Akemann and Weaver~\cite{AkeWea} 
answered  this to the negative.
We shall prove that this $C^{\ast}$-algebra does not have the 
weak fixed point property. This shows that the separability assumption in our
theorem is essential.
\par
Section~\ref{sec:wfpp} contains our main theorem and its proof. In
section~\ref{sec:kk}
we consider the uniform weak$^{\ast}$ Kadec-Klee property
(see definition~{\ref{def:ukk}) 
of the Banach space  dual of the $C^*$-algebras in question.
For the trace class operators this property holds true as proved by 
Lennard~\cite{Len90}. As he points out, the weak$^{\ast}$ fixed point property
can be obtained, via weak$^{\ast}$ normal  structure of 
non-empty weakly$^{\ast}$
compact sets
by an application of UKK$^{\ast}$, as shown by van Dulst and Sims~\cite{DuSi83}.
For corresponding results with the weak topology 
we refer to the article by Kirk~\cite{Kirk65} and that by Lim~\cite{Lim80}.
\section{Weak$^{\ast}$ Fixed Point Property}\label{sec:wfpp}
In this section $A$ shall be a separable $C^{\ast}$-algebra, unless stated otherwise. 
We denote by $\pi^{'}\simeq\pi$ unitary equivalence of 
$*$-representations $\pi^{'}$ and $\pi$. By abuse of notation we denote by
$\pi$ 
also its equivalence class.
\par
The following proposition is based on a theorem of Anderson~\cite{And}, which 
itself refines a lemma of Glimm~\cite[Lemma 9]{Glimm2} and
 \cite[Theorem 2]{Glimm1}. 
\begin{prop}\label{prop:comp}Let $A$ be a separable $C^{\ast}$-algebra.
Let $\pi^{'}\not\simeq\pi\in \widehat{A}$ with $\pi^{'}\in\overline{\{\pi\}}$
be given and assume that $\varphi$ is a state of $A$ associated with $\pi^{'}$.
Then there is an orthonormal sequence $(\xi_n)$ in $H_{\pi}$ with
$(\pi(.)\xi_n|\xi_n)\to \varphi$ $\text{weakly}^{\ast}$.
\end{prop}
\begin{proof}
By assumption, $\ker{\pi^{'}}\supset\ker{\pi}$ so there is a representation 
$\pi^{\circ}$ of $\pi(A)$ such that $\pi^{'}=\pi^{\circ}\circ\pi$.
We may therefore assume that $\pi$ is the identical representation.
We denote $\mathcal{K}(H)$ the $C^{\ast}$-algebra of compact operators
on the Hilbert space $H$.
\begin{trivlist}
\item[(i)]
Suppose 
$\varphi_{\vert \mathcal{K}{(H_{\pi})}\cap A}\not= \emptyset$. Then 
$\pi^{'}=\pi_{\varphi}$
does not annihilate 
$\mathcal{K}(H_{\pi})\cap A$. By \cite[Corollary 4.1.10]{Dix} 
$\mathcal{K}{(H_\pi)}\subset A$ and it is a two sided ideal.
This corollary does not cover our case completely but we 
follow its proof.
$\pi^{'}$
is faithful on 
$\mathcal{K}{(H_{\pi})}$ and $\pi^{'}_{\vert\mathcal{K}{(H_{\pi})}}$
is an irreducible representation by \cite[2.11.3]{Dix}.
Therefore it is equivalent to the identical representation of 
$\mathcal{K}{(H_{\pi})}$.
Now $\pi^{'}$ is equivalent to the identical representation 
of $A$,
by \cite[2.10.4(i)]{Dix}.
This contradicts the assumption, so this case can not happen.
\item[(ii)]
If $\varphi_{\vert \mathcal{K}{(H_{\pi})}\cap A}= 0$.
then by \cite[Theorem]{And} there is an orthonormal sequence
$(\xi_n)$ in $H_{\pi}$ with
$(\pi(.)\xi_n|\xi_n)\to \varphi$ $\text{weakly}^{\ast}$.\qedhere
\end{trivlist}
\end{proof}
\begin{lemma}\label{lem:type}
Let $M$ be a von Neumann algebra. If its predual $M_{\ast}$.
has the weak fixed point property, then 
$M$ is of type I. Moreover $M$ is atomic.
\end{lemma}
\begin{proof}
The argument follows the proof of ~\cite{Rand10}{ Theorem 4.1}.
We denote by $\mathcal{R}$ the hyperfinite factor of type II$_1$
and $\tau_{\mathcal{R}}$ its canonical finite trace 
(see e.g.\ \cite{Tak02}).
In ~\cite{MarNha} it is proved that its predual 
$L^1(\mathcal{R},\tau_{\mathcal{R}})$ embeds isometrically into the predual
of any von Neumann algebra not of type I.
As $L^1([0,1],\leb)$ embeds isometrically into 
$L^1(\mathcal{R},\tau_{\mathcal{R}})$ (\cite{LauLeinert08}{ Lemma 3.1})
we conclude from Alspach's theorem~\cite{Alsp} that the weak 
fixed point property of $M_{\ast}$ forces $M$ to be a 
type I von Neumann algebra.
So $M$ has a normal semifinite faithful trace~\cite[A35]{Dix}.
Now,
\cite{LauLeinert08}{ Proposition 3.4}
implies that $M$ is an atomic von Neumann algebra.
\end{proof}
Lemma~\ref{lem:type} and Proposition 3.4 of~\cite{LauLeinert08}
provide the converse to~\cite[Lemma 3.1]{LauMahUlg}
and thus answers a question of
A.~T.-M.~Lau~\cite[Problem 1]{Lau-AHA13}:
\begin{corollary}
Let $M$ be a von Neumann algebra, then $M_{\ast}$ has the 
weak fixed point property if and only if it has the Radon--Nikodym property.
\end{corollary}
\begin{remark}
If now $A$ is a $C^{\ast}$-algebra whose Banach space dual $A^{\ast}$ has the 
weak fixed point property, then, by Lemma~\ref{lem:type}, 
$A^{\ast\ast}$ is a type I von Neumann algebra
and
we know from ~\cite[6.8.8]{Ped}
that $A$ is a 
type I $C^{\ast}$-algebra. Especially, its spectrum,
which coincides with the space of its primitive ideals in this case, 
is a $T_0$ topological space. A fortiori, this also holds if $A^{\ast}$ has the
weak$^{\ast}$ fixed point property.
\end{remark}
\begin{prop}\label{prop:closed}
Assume that $A$ is separable.
If $A^{\ast}$ has the weak$^{\ast}$ fixed point property then
points in $\widehat{A}$ are closed.
\end{prop}
\begin{proof}
If $\{\pi\}$ is non--closed in $\widehat{A}$ then there
is $\pi^{'}\not\simeq \pi$ contained in $\overline{\{\pi\}}$.
By  Proposition~\ref{prop:comp}, if 
$\varphi$ is a (pure) state associated with $\pi^{'}$, then there exists
an orthonormal sequence $(\xi_n)$ in $H_{\pi}$ such that
$\varphi_n:= (\pi(.)\xi_n|\xi_n)\to \varphi$ weakly$^{\ast}$.
Now we proceed as in \cite{FLL}. Set 
$\varphi_0=\varphi$, then the set 
$$C=\{\;\sum_0^{\infty} \alpha_i \varphi_i\;:\;%
0\leq \alpha_i\leq 1, \sum_0^{\infty} \alpha_i =1\}$$
is convex weak$^{\ast}$ compact.
The coefficients of every $f=\sum_0^{\infty} \alpha_i \varphi_i\in C$
are uni\-que\-ly determined: 
\par
Since $A^{\ast}$ is assumed to have the weak$^{\ast}$ fixed point property,
the universal enveloping von Neumann algebra  $A^{\ast\ast}$ of $A$ is atomic.
By~\cite[Appendix A]{FLL}, the universal representation of $A$ 
decomposes into a direct
sum of irreducible representations. 
Hence we may apply \cite{FLL}{ Lemma 4.2\footnote{
From the context there, one sees that there it is assumed 
that every $\ast$-representation of the $C^{\ast}$-algebra in question 
decomposes into a direct {H}ilbert sum of irreducible reprersentations}}
to see that the support 
$P_0$ of $\varphi_0$
(in the universal enveloping von Neumann algebra)
is orthogonal to the 
support of every  other $\varphi_i$.
So, denoting the ultraweak extensions to $A^{\ast\ast}$ of $f\in C$ and
$\varphi_i,\;i\geq 0$, by the same symbols again, we have
$f(P_0)=\alpha_0\varphi_0(P_0)=\alpha_0$.
It remains to pick out the remaining $\alpha_i$ from the sum
$\sum_{1}^{\infty}\alpha_i \varphi_i$.
Since $\pi$ is irreducible, its ultraweak extension to 
$A^{\ast\ast}$ has $B(H_{\pi})$ as its range.
So we can evaluate 
the sum $\sum_{1}^{\infty}\alpha_i\varphi_i$
at $P_n$, the 1-dimensional projection onto $\CC\cdot\xi_n$,
which yields exactly $\alpha_n$.
\par
Now we may define $T: C\to C$ by
$$T\left(\sum_0^{\infty} \alpha_i \varphi_i\right)=
\sum_0^{\infty} \alpha_i \varphi_{i+1}.$$
To show that this map is distance preserving
it suffices to see that 
$\norm{\sum_0^{\infty} \beta_i \varphi_{i}} = \sum_0^{\infty} \abs{\beta_i}$,
for  real summable $\beta_i$.
Clearly $\norm{\sum_0^{\infty} \beta_i \varphi_{i}}= \abs{\beta_0}+
\norm{\sum_1^{\infty} \beta_i \varphi_{i}}$, since the support of 
$\varphi_0$ is orthogonal to the support of any $\varphi_i,\;i\geq 1$.
Since $B(H_{\pi})=A^{\ast\ast}/{\ker(\pi)}$ isometrically, the norm
$\norm{\sum_1^{\infty} \beta_i \varphi_{i}}$ can be calculated in $B(H_{\pi})$.
\par
The element
$Q=\sum_1^{\infty}\sign(\beta_i) P_i\in B(H_{\pi})$ has norm $1$ and
$\sum_1^{\infty} \beta_i \varphi_{i}(Q)=\sum_1^{\infty} \abs{\beta_i}$.
So $\norm{\sum_0^{\infty} \beta_i \varphi_{i}}\geq \sum_0^{\infty} \abs{\beta_i}$.
In fact equality holds, since the reverse inequality is plain.
Hence $T$ is distance preserving.
The definition of $T$ is such that the only possible fixed point would be $0$.
But $0\notin C$ and we arrive at a contradiction.
\end{proof}
\begin{theorem}\label{thm:wfpp}
For a separable $C^{\ast}$-algebra the following are equivalent
\begin{itemize}
\item[(i)]The spectrum $\widehat{A}$ is discrete.
\item[(ii)]$A^{\ast}$ has the weak$^{\ast}$ fixed point property.
\item[(iii)]$A^{\ast}$ has the weak$^{\ast}$ fixed point property
for left reversible semigroups.
\end{itemize}
\end{theorem}
\begin{proof}
We assume that  $A^{\ast}$ has the weak$^{\ast}$ fixed point property.
If  $\widehat{A}$ is not discrete, then there is some point $\pi_0\in\widehat{A}$
which is in the closure of some set $M\subset\widehat{A}$
not containing $\pi_0$.
Because of the last proposition $M$ must be infinite. 
So, since $A$ is separable, for any state $\varphi_0$ associated to $\pi_0$ 
there is a sequence of states 
$(\varphi_n)$ associated to pairwise non-equivalent representations $\pi_n$ 
with $\varphi_n\to \varphi_0$ weakly$^{\ast}$. 
By \cite[Lemma 4.2]{FLL} the support projections in the 
universal  representation of $A$ are mutually orthogonal. 
As in the proof of proposition~\ref{prop:closed} the set 
$C=\{\sum_0^\infty \alpha_i\varphi_i : \alpha_i\geq 0, \sum\alpha_i=1\}$
is convex and weak$^{\ast}$ compact. The map $T:C\to C$ defined like there is
well defined and isometric because of the orthogonality of the supports,
and it has no fixed point in $C$. See also \cite[Theorem 4.5]{FLL}.
So $\widehat{A}$ must be discrete.
\par
Conversely, if $A$ has a discrete spectrum then the Jacobson topology
on its set of
primitive ideals is discrete too. By \cite[10.10.6 (a)]{Dix} $A$
is a $c_0$-direct sum of $C^{\ast}$-algebras with one point spectrum.
Since $A$ is separable these algebras all are separable too 
(the sum is on a countable index set of course),
and hence each of them is isomorphic  to an algebra of compact operators 
on some
Hilbert space (of at most countable dimension) \cite[4.7.3]{Dix}. 
By \cite[Corollary 3.7]{Rand10}
we have that $A^{\ast}$ has the weak$^{\ast}$ fixed point property for left reversible semigroups. By specialisation the 
weak$^{\ast}$ fixed point property for single non-expansive mappings follows. 
\end{proof}
\begin{remark}
If one enriches the ZFC set theory with the diamond axiom
then there is a non-separable $C^{\ast}$-algebra
with discrete spectrum, whose Banach space dual does not possess the weak
(and a fortiori not the weak$^{\ast}$) fixed point property.
\end{remark}
\begin{proof}
The $C^{\ast}$-algebra $A$ constructed by Akemann and Weaver~\cite{AkeWea}
is not a type I $C^{\ast}$-algebra, but it has a 
one point spectrum.
It follows also in the non-separable case that $A^{\ast\ast}$ is not 
a type I von Neumann algebra (see ~\cite{Ped}{ 6.8.8}).
By Lemma~\ref{lem:type} we obtain our assertion.\end{proof}
\section{Uniform Weak$^{\ast}$ Kadec-Klee}\label{sec:kk}
Let $K\subset E$ be a closed convex bounded subset of a Banach space $E$.
A point $x\in K$ is a {\em diametral} point if 
$\sup\{\norm{x-y}:y\in K\}=\diam{K}$. 
The set $K$ is said to have {\em normal structure} if every convex
non-trivial (i.e.\ containing at least two different points) subset 
$H\subset K$ 
contains a non-diametral point of $H$.
\par
A Banach space has {\em weak normal structure} if every convex weakly 
compact subset has normal structure, and similarly a
dual Banach space has {\em weak$^{\ast}$ normal structure}
if every convex weakly$^{\ast}$ compact subset has normal structure.
\par
A dual Banach space $E$ is said to have the {\em weak$^{\ast}$ Kadec-Klee 
property} (KK$^{\ast}$) if weak$^{\ast}$ and norm convergence coincide on sequences of
its unit sphere.
\begin{definition}\label{def:ukk}
A dual Banach space $E$ is said to have the {\em uniform weak$^{\ast}$ 
Kadec-Klee  property} (UKK$^{\ast}$) if for $\epsilon>0$ there is $0<\delta<1$
such that for any subset $C$ of its closed unit ball 
containing an infinite sequence $(x_i)_{i\in\NN}$
with separation $\sep{(x_i)_i}:= \inf\{\norm{x_i-x_j}:i\not= j\}>\epsilon$,
there is an $x$ in the weak$^{\ast}$-closure of $C$ with $\norm{x}<\delta$.
\end{definition}
For a discussion of these and similar properties we refer the 
interested reader to~\cite{LauMah}.
The following proposition is known, but we could not find a valid reference. 
So, for the reader's convenience,
we give a proof.
\begin{prop}\label{prop:sphere-top}
Let $E$ be a dual Banach space.
\begin{itemize}
\item[(i)] The uniform weak$^{\ast}$ Kadec-Klee  property
implies the weak$^{\ast}$ Kadec-Klee  property.
\item[(ii)]If $E$ is the dual of a separable Banach space $E_{\ast}$ and has the 
uniform weak$^{\ast}$ Kadec-Klee  property
then the weak$^{\ast}$ topology and the norm topology coincide 
on the unit sphere of $E$.
\end{itemize}
\end{prop}
\begin{proof}
To prove (i) assume that $\norm{x_n}=1$, $x_n \to x$ weakly$^{\ast}$ 
and $\norm{x}=1$.
If $\{x_n:n\in\NN\}$ is relatively norm compact then 
the only norm accumulation point has to be $x$,
since the norm topology is finer than the weak$^{\ast}$ topology,
and a subsequence has to converge in norm to $x$.
We hence assume that $\{x_n:n\in\NN\}$ is not relatively compact in the  norm 
topology and  shall derive a contradiction.
Using that  $\{x_n:n\in\NN\}$ is not totally bounded, by induction, we obtain a subsequence $(x_{n_k})_k$ with $\sep{(x_{n_k})_k}>0$.
By the UKK$^{\ast}$ property
there is $\delta<1$ and a weak$^{\ast}$ accumulation point
$y$ of $(x_{n})_n$ with $\norm{y}<\delta<1$.
Since every weak$^{\ast}$ neighbourhood of $y$ contains
infinitely many $x_n$, it follows that $y=x$. 
This contradicts $\norm{x}=1$.
\par
Now (ii) follows since in this case the weak$^{\ast}$ topology on 
the unit sphere of $E$ is metrisable.
\end{proof}
\begin{theorem}\label{thm:sphere}
For a separable $C^{\ast}$-algebra $A$ the 
following are equivalent
\begin{itemize}
\item[(i)]The spectrum $\widehat{A}$ is discrete,
\item[(ii)]The Banach space dual $A^{\ast}$ has the UKK$^{\ast}$ property,
\item[(iii)]On the unit sphere of $A^{\ast}$
the weak$^{\ast}$ and the norm topology coincide,
\item[(iv)]On the set of states $\mathcal{S}(A)$ of $A$ the weak$^{\ast}$ and the norm topology coincide,
\item[(v)]On the set of pure states $\mathcal{P}(A)$ of $A$ 
the weak$^{\ast}$ and the norm topology coincide.
\item[(vi)]$A^{\ast}$ has weak$^{\ast}$ normal structure.
\item[(vii)]$A^{\ast}$ has the  weak$^{\ast}$ fixed point property for 
non-expansive mappings.
\item[(viii)]$A^{\ast}$ has the  weak$^{\ast}$ fixed point property for 
left reversible semigroups.
\end{itemize}
\end{theorem}
\begin{remark}
The $C^{\ast}$-algebras fulfilling the equivalent conditions of 
the theorem are just the separable dual $C^{\ast}$ algebras, 
see~\cite[p. 157]{Tak02}
for the definition. This follows from the fact that separable
dual $C^{\ast}$-algebras are characterised by the property that their
spectrum is discrete~\cite[9.5.3 and 10.10.6]{Dix} 
see also~\cite[p. 706]{Kus01}.
\end{remark}
\begin{proof}[Proof of theorem \ref{thm:sphere}.]
Assume (i) then, as in the proof of Theorem~\ref{thm:wfpp},
$A^*$ is a countable $l^1$-direct sum of trace class operators in 
canonical duality to
the corresponding $c_0$-direct sum of compact operators.
Moreover, considering $A^*$ as block diagonal trace class operators
on the Hilbert space direct sum of the underlying Hilbert spaces gives
an isometric embedding of $A^*$ in the trace class operators on this
direct sum Hilbert space. The image is closed in the weak$^{\ast}$ topology
and we obtain the UKK$^{\ast}$ property of $A^*$ from
 the UKK$^{\ast}$ property of the trace class operators~\cite{Len90}.
\par
Now (ii) implies (iii) by the above Proposition~\ref{prop:sphere-top}.
Clearly, (iii) implies (iv) and the latter implies (v) by restriction.
So the first part of our proof will be finished by proving the 
implication (v)$ \implies$ (i).
We adapt the proof of \cite[Lemma 3.7]{BKLS98} to our context.
For $\varphi\in\mathcal{S}(A)$ denote $\pi_{\varphi}$ the representation 
of $A$ obtained from the GNS-construction. Here extreme points yield 
irreducible representations and
conversely a representative of any element of $\widehat{A}$ can be 
obtained in this way. Moreover, if $\mathcal{P}(A)$ is endowed with
the weak$^{\ast}$ topology then  the mapping 
$q:\varphi\to \pi_{\varphi}$ is open~\cite[Theorem 3.4.11]{Dix}.
By \cite[Corollary~10.3.8]{KadRing86}  for   
$\varphi,\psi\in \mathcal{P}(A)$
the  representations
$\pi_{\varphi}$ and $\pi_{\psi}$ are equivalent if $\norm{\varphi-\psi}<2$
(see also \cite[Corollary 9]{GliKad}).
Hence, assuming (v), the (norm open) set
$\{\psi\in \mathcal{P}(A):\norm{\varphi-\psi}<2\}$ is a weak$^{\ast}$ open 
neighbourhood of $\varphi$ in $\mathcal{P}(A)$.
Its image under $q$ is open but just reduces to the point $\pi_{\varphi}$.
This shows that points in $\widehat{A}$ are open.
\par
Now (ii) $\implies$ (vi) is proved in \cite{LauMah}, 
(vi) $\implies$ (vii) is proved in \cite{Lim80} (see also \cite{DuSi83}).
(vii) $\implies $ (i) holds true by Theorem~\ref{thm:wfpp}.
From this theorem we have (viii) $\iff$ (i) too.
\end{proof}
\begin{remark}Each of the following conditions implies {(i)--(viii)} above
and, if $A$ is the group $C^{\ast}$-algebra of a locally 
compact group $G$, is equivalent to them. (See ~\cite[Section 5]{FLL}
for the definitions involved.)
\begin{itemize}
\item[(ix)]$A^{\ast}$ has the $\limsup$ property.
\item[(x)]$A^{\ast}$ has the asymptotic centre property.
\end{itemize}
\end{remark}
Under the assumption of separability it is shown in
\cite[Theorem 4.1]{LauMah10} that the $\limsup$ property implies  the asymptotic centre property. From this in turn the 
weak$^{\ast}$ fixed point property for left reversible semigroups
follows~\cite[Theorem 4.2]{LauMah10}.
Without the separability these implications hold equally true,
see \cite{FLL}).
\begin{remark}
It is proved in ~\cite[Theorem 5]{LauMah86}
that the $\limsup$ property, which is equivalent to  Lim's condition  
considered there,
is not fulfilled in the space of trace 
class operators of an infinite dimensional Hilbert space $H$.
So for $A=\mathcal{K}(H)$ the $\limsup$ property for $A^{\ast}$ is not satisfied 
and hence not equivalent to {(i)--(viii)} above.
It seems unlikely that the asymptotic centre property holds true in this case. 
\end{remark}
\section{Acknowledgements}
{Part of this work was done, when the second named author
visited the Erwin Schr\"odinger Institute at Vienna in autumn 2012.
He gratefully acknowledges the institute's  hospitality.
We thank Vern Paulsen and Ian Raeburn for fruitful conversations and 
Vern Paulsen for pointing out
reference \cite{And} to us. We thank Anthony Lau for
pointing out the appearance 
of dual $C^{\ast}$-algebras in our theorem~\ref{thm:sphere}.}
\def\cprime{$'$}

\end{document}